\documentclass[12pt]{article}
\usepackage[utf8]{inputenc}
\usepackage{tabularx}

\usepackage{amsfonts,amsmath,amssymb,amsthm,amscd}
\usepackage{authblk}

\usepackage{graphicx}
\usepackage{caption}
\usepackage{subcaption}
\usepackage{hyperref}
\usepackage{courier}
\usepackage{fancybox}
\usepackage{fancyhdr}
\usepackage{float}
\usepackage{alltt}
\usepackage{fancyvrb}

\usepackage[dvipsnames,usenames]{color}
\usepackage[shortlabels]{enumitem}

\usepackage{framed}
\usepackage{multirow,array}
\usepackage{booktabs}
\usepackage{bm}

\newtheorem{corollary}{Corollary}
\newtheorem{definition}{Definition}
\newtheorem{lemma}{Lemma}
\newtheorem{proposition}{Proposition}
\newtheorem{remark}{Remark}
\newtheorem{theorem}{Theorem}
\newtheorem{claim}{Claim}

\usepackage{xcolor}
\newcounter{teoA}
\newtheorem{teoa}[teoA]{Theorem}

\providecommand{\keywords}[1]
{
  \small	
  \textbf{\textit{Keywords:}} #1
}
\providecommand{\subjclass}[1]
{
  \small	
  \textbf{\textit{2020 Mathematics Subject Classification:}} #1
}

\title{On the accumulation points of non-periodic orbits of a difference equation of fourth order}
\author{A. Linero Bas\footnote{Universidad de Murcia, Departamento de Matem\'{a}ticas, Campus de Espinardo, 30100 Murcia (Spain). \textit{E-mail address:} lineroba@um.es}, V. Mañosa\footnote{Universitat Politècnica de Catalunya, Departament de Matemàtiques, Institut de Matem\`atiques de la UPC (IMTech). Colom 1, 08222 Terrassa (Spain). \textit{E-mail address:} victor.manosa@upc.edu}, D. Nieves\footnote{Universidad de Murcia, Departamento de Matem\'{a}ticas, Campus de Espinardo, 30100 Murcia (Spain). \textit{E-mail address:} daniel.nieves@um.es}}
\date{}

\begin{document}

\maketitle

\begin{abstract}
In this paper, we are interested in analyzing the dynamics of the fourth-order difference equation 
$x_{n+4} = \max\{x_{n+3},x_{n+2},x_{n+1},0\}-x_n$, with arbitrary real initial conditions.
We fully determine the accumulation point sets of the non-periodic solutions that, in fact, are configured as proper compact intervals of the real line. This study complements the previous knowledge of the dynamics of the difference equation already achieved in
[\textit{M. Cs\"{o}rnyei, M. Laczkovich}, Monatsh. Math. \textbf{132} (2001), 215-236] and 
[\textit{A. Linero Bas, D. Nieves Roldán,} J. Difference Equ. Appl. \textbf{27} (2021), no. 11, 1608-1645].

\end{abstract}



\vfill

\section{Introduction}

By a difference equation of max-type, we understand an autonomous or non-autonomous difference equation whose solutions are generated by a recurrence law involving the max operator. There exist a large literature dealing with different classes of max-type equations, in which the main interest is to know the dynamics at large of the orbits generated by the recurrence, and try to apply these equations for the modelling of processes appearing in fields as the biology, economy, control theory,...; for more information, consult book \cite{Gro} or the survey \cite{LNsurvey}, and references therein.

One of the most popular difference equations of this class is the max-type version of Lyness equation, given by $x_{n+2}={\max\{1,x_{n+1}\}}/{x_{n}}$, or its generalization in the form $x_{n+2}={\max\{A,x_{n+1}^k\}}/{x_{n+1}^{\ell}x_{n}}$, where $A, k, \ell$ are real coefficients. In particular, in the case of positive initial conditions, by the change of variables $y_n=\ln x_n$, the difference equation $x_{n+2}={\max\{1,x_{n+1}\}}/{x_{n}}$ is transformed into $y_{n+2}=\max\{y_{n+1}, 0\}-y_{n}$; the dynamics of both equations is simple: all the solutions are periodic of period $5$ (not necessarily minimal). A natural generalization to the above equation is given by 
\begin{equation}\label{e:intro}
y_{n+k}=\max\{y_{n+k-1},\ldots,y_{n+2},y_{n+1}, 0\}-y_{n}.
\end{equation}

If $k=3$, it is easily seen that all the solutions of Equation \eqref{e:intro} are periodic of period $8$ (not necessarily minimal), so the dynamics remains to be simple. 

In this paper,  we will go a step further  and will focus on the fourth-order difference equation
\begin{equation} \label{Eq_main}
    x_{n+4} = \max\{x_{n+3},x_{n+2},x_{n+1},0\}-x_n,
\end{equation}

\noindent with arbitrary real initial conditions. 

In \cite{Lin}, the periodic character of Equation (\ref{Eq_main}) was deeply studied, and the authors were able to describe precisely its set of periods, as well as its associate periodic orbits. Nevertheless, nothing was said about the non-periodic orbits beyond their existence and that they are bounded, a result that was already obtained in 
\cite{Cso}. In this sense, a natural question that arises is to analyze the dynamics of such orbits. In particular, we will focus on the set of accumulation points of a solution $(x_n)$ under the iteration of Eq.~(\ref{Eq_main}), whenever the non-periodic character is ensured. Thus, the present work can be seen as a natural continuation of~\cite{Lin}.

Based on \cite{Lin}, where the authors analyzed the possible configurations of the initial conditions in order to obtain periodic sequences, it is known that the trajectory of a tuple of non-negative initial conditions  $(x_1,x_2,x_3,x_4)$, where $x_1=\max_n\{x_n\}\geq 0$, satisfies some of the conditions of five cases $C_i$, $i=1,\ldots, 5$ detailed in Figure \ref{Diagrama} and Table~\ref{Table_cases}. The orbit of a solution visits these cases following a series of routes composed of successive cases $C_j$ that are applied according to several conditions about the ordering of the elements of the tuple.

\begin{figure}[ht] 
\centering
\includegraphics[scale=0.45]{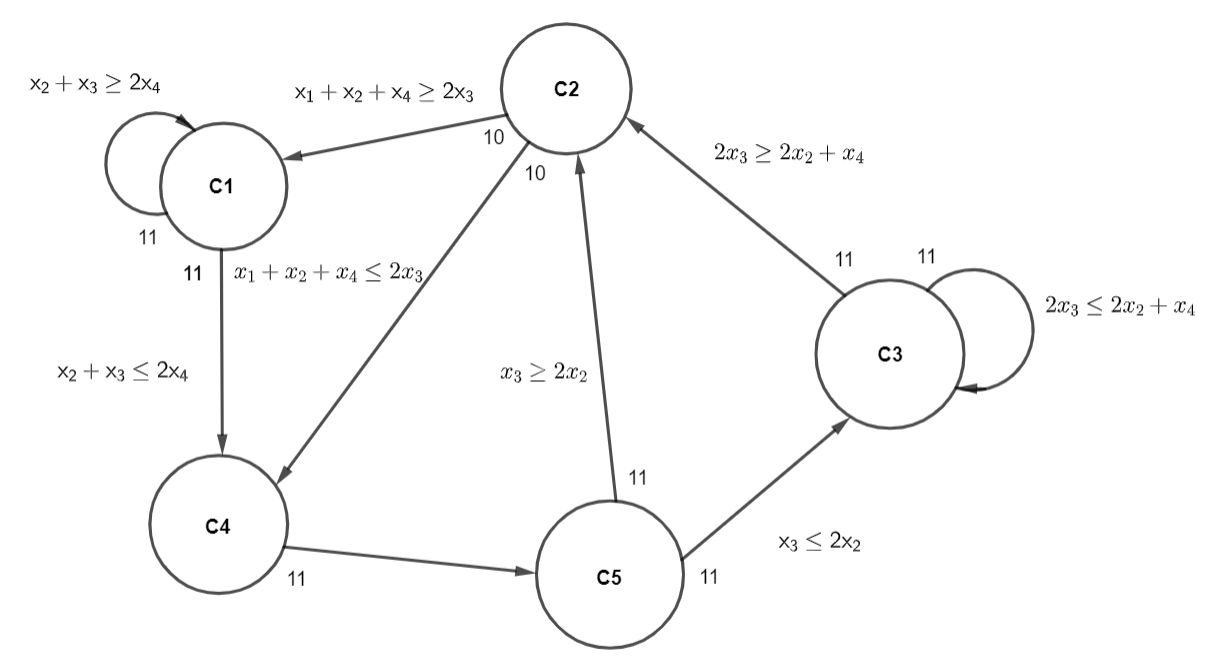}
\caption{Evolution of the cases under the iteration of  Equation~(\ref{Eq_main}).}
\label{Diagrama}
\end{figure}

For the sake of completeness, we add Table \ref{Table_cases} from \cite{Lin} in order to show how a tuple $(x_1,x_2,x_3,x_4)$ evolves under Equation (\ref{Eq_main}) when it satisfies the conditions of each case. It should be mentioned that the information collected in the table follows by the inspection of the proof of \cite[Proposition 12]{Lin}.

\begin{table}[ht] 
\begin{tabular}{@{}cccc@{}}
\toprule
\textbf{$C_1:$} & $x_1 \geq x_2 \geq x_4 \geq x_3\geq 0$ & $\xrightarrow{11}$ & $(x_{12}= x_1, x_2+x_3-x_4, x_3, x_4)$ \\ \midrule
\textbf{$C_2:$} & \begin{tabular}[c]{@{}c@{}}$x_1 \geq x_3 \geq \max\{x_2, x_4\}\geq 0$\\ $x_3 \geq x_2 + x_4$\end{tabular}  & $\xrightarrow{10}$ & $(x_{11}= x_1, x_1-x_3+x_2+x_4, x_2, x_3)$ \\ \midrule
\textbf{$C_3:$} & \begin{tabular}[c]{@{}c@{}}$x_1 \geq x_3 \geq \max\{x_2, x_4\}\geq 0$ \\ $x_3 \leq x_2 + x_4$\end{tabular} & $\xrightarrow{11}$ & $(x_{12}= x_1, x_2, x_3, x_2+x_4-x_3)$ \\ \midrule
\textbf{$C_4:$} & $x_1 \geq x_4 \geq x_2 \geq x_3\geq 0$ & $\xrightarrow{11}$ & $(x_{12}= x_1, x_3, x_4+x_3-x_2, x_4)$ \\ \midrule
\textbf{$C_5:$} & $x_1 \geq x_4 \geq x_3 \geq x_2\geq 0$ & $\xrightarrow{11}$ & $(x_{12}= x_1, x_2, x_4, x_2+x_4-x_3)$     \\ \bottomrule
\end{tabular}
\caption{Evolution of a tuple $(x_1,x_2,x_3,x_4)$ in the different cases.}
\label{Table_cases}
\end{table}

For instance, the tuple $(x_1,x_2,x_3,x_4)=(3,1,\sqrt{2},2)$ satisfies the conditions of Case $C_5$, namely $x_1\geq x_4\geq x_3\geq x_2\geq 0$, and after $11$ iterations it evolves to the tuple $(x_{12},x_{13},x_{14},x_{15})=(3,1,2,3-\sqrt{2})$ that satisfies now the conditions of Case~$C_3$.

Prior to state the main result, let us recall the definition of an equivalence relation established in \cite{Lin}. Notice that, by Eq. (\ref{Eq_main}), for given initial conditions, we can build a unique sequence $(x_n)_{n\in \mathbb{Z}}$ (realize that backward terms are obtained in a one-to-one way through the relation $z_{n}=\max\{z_{n+1}, z_{n+2},z_{n+3},0\}-z_{n+4}$).

\begin{definition} \label{Def_equi}
Set $\bm{x}= (x_1,x_2,x_3,x_4)$, $\bm{y}= (y_1,y_2,y_3,y_4) \in \mathbb{R}^4$. We will say that they are equivalent, and write $\bm{x} \sim \bm{y} $, if and only if $\bm{x}$ and $\bm{y}$ generate under Eq. (\ref{Eq_main}) the same sequences $(x_n)_{n\in \mathbb{Z}}$ and $(y_n)_{n\in \mathbb{Z}}$ up to a shift.
\end{definition}

The main result of the paper establishes that every non-periodic solution $(x_n)_n$ of Equation~(\ref{Eq_main}) is dense in a compact interval of the real line.

\begin{teoa}\label{t:teoprincipal}
Given real initial conditions $(x_1,x_2,x_3,x_4)$ that generate a non-periodic orbit under Equation~(\ref{Eq_main}), its set of accumulation points is a compact interval. Even more,  the tuple $(x_1,x_2,x_3,x_4)$ is equivalent to some tuple of initial conditions $(x,y,z,w),$ with $x=\max\{x_n:n\geq 1\}$, $x\geq w \geq y\geq z \geq 0$, and $\frac{w-z}{x} \in \mathbb{R}\setminus\mathbb{Q}$, and the orbit accumulates in the compact interval $[\min\{w-x,-z\},x]$.
\end{teoa}

The paper is organized as follows: in Section \ref{Sec_prelim} we establish some preliminaries, in concrete, from the boundedness character of every solution of Equation (\ref{Eq_main}), we derive that we can assume, without loss of generality, that the first term of the sequence is the maximum of the orbit, regardless of whether the orbit is periodic or not. Moreover, if we consider initial conditions that generate a non-periodic orbit, we will show that such solution is surrounded by periodic solutions of arbitrarily large periods. 

Section \ref{Sec_condU} describes the evolution  of non-negative tuples $(x_1,x_2,x_3,x_4)$, with $x_1=\max_n\{x_n\}$, starting from Case $C_4$. We use this description to obtain the accumulation points of the solution $(x_n)$ in Section~\ref{S:mainproof}. First, we prove that the non-negative elements of the solution accumulate in the interval $[0,x_1]$; second, after proving that the initial tuple visits all the five Cases $C_j$, $1\leq j\leq 5$, infinitely many times, we obtain that the non-positive terms of the solution are dense in the interval $[\min\{x_4-x_1, -x_3\},0].$ The above study will provide us immediately the proof of Theorem~\ref{t:teoprincipal}.  Finally, we present some comments and observations for further research in Section \ref{Sec_conclusions}. In concrete, we find a new first integral of the discrete dynamical system associated to Equation \eqref{Eq_main} and comment on the possible existence of another first integral for the system based on numerical simulations. 

\section{Preliminaries} \label{Sec_prelim}

Firstly, we comment on the boundedness character of the solutions of Eq. (\ref{Eq_main}). Let us consider a sequence generated by the initial conditions $(x_1,x_2,x_3,x_4)$ under such equation. 

By \cite{Cso}, we know that every sequence generated by Eq. (\ref{Eq_main}) is bounded. Furthermore, it is proved there that $$|x_n| \leq M:=\max\{|x_j|: 1 \leq j \leq 12\} \ \text{for  all} \ n\geq 1. $$

It is easy to see that,  in fact, the bound $M$ is attached by the maximum of the positive terms in the sequence, that is $M = \max\{x_n\} \geq 0$. Indeed, take $x_j$ such that $|x_j|=M$, then:
\begin{itemize}
    \item[$\bullet$] If $x_j \geq 0$, the result follows.
    \item[$\bullet$] If $x_j < 0$, we can find the following terms in the sequence generated by the recurrence: $\ldots, -|x_j|, x_{j+1}, x_{j+2}, x_{j+3}, \ldots$ So,
    \begin{eqnarray*}
    x_{j+4} &=& \max\{x_{j+1}, x_{j+2}, x_{j+3}, 0\} - (-|x_j|) \\
    &=& \max\{x_{j+1}, x_{j+2}, x_{j+3}, 0\} + |x_j| \geq M \geq 0.
    \end{eqnarray*}
    This implies $x_{j+4} = M$ and we have reached the maximum of the sequence with a positive term. 
\end{itemize}

According to the equivalence relation given in Definition \ref{Def_equi}, we can state, without loss of generality, the following restriction. 
\begin{claim} \label{Claim0}
We can assume that $x_1=\max\{x_n:n\geq 1\}$ for every sequence $(x_n)$ generated by Equation (\ref{Eq_main}). 
\end{claim}
Indeed, from the above analysis, if $j_{max}$ is the index that verifies $$x_{j_{max}} = \max\{|x_j|:1\leq j\leq 12\} \geq |x_n| \ \ \text{for \ all} \ n\geq 1,$$

\noindent we can take the shifted sequence generated by $y_1=x_{j_{max}}$; $y_2=x_{j_{max}+1}$; $y_3=x_{j_{max}+2}$; $y_4=x_{j_{max}+3}$, which, by Definition \ref{Def_equi}, will generate the same solution as $(x_1,x_2,x_3,x_4)$ under Eq. (\ref{Eq_main}). $\hfill\Box$

Therefore, 
in the sequel we will consider that $x_1=\max\{x_n:n\geq 1\}$.   Notice that, necessarily, this implies that $x_1\geq 0$ and also that $x_2, x_3,$ and $x_4$ are non-negative terms. 

Next, we will center our study in the evolution of non-periodic orbits generated by initial conditions $(x_1,x_2,x_3,x_4)$ under Eq.~(\ref{Eq_main}).
The following result gives a necessary and sufficient condition in order to have a non-periodic orbit. It is a direct consequence of the argument that gives rise to Eq. (10) in \cite[p. 28]{Lin}.
\begin{proposition}\label{P:noper}\cite{Lin}
Let $(x_1,x_2,x_3,x_4)$ be a tuple of initial conditions, with $x_1=\max_n\{x_n\}$, and holding the restrictions of Case $C_4$. Let $(x_n)$ be the corresponding solution generated by Equation~(\ref{Eq_main}). Then, the solution is non-periodic if and only if $\frac{x_4-x_3}{x_1}\not\in\mathbb Q.$
\end{proposition}
We remark that requiring  in the above result that the conditions of Case $C_4$ are satisfied  does not imply a loss of generality, as demonstrated in \cite{Lin}, and as also can be deduced from the contents of Section \ref{Sec_condU}.

The following result establishes that each initial condition leading to a non-periodic sequence has, arbitrarily close, initial conditions leading to periodic sequences, and that the set of periods of these solutions is not bounded.

\begin{proposition}
Let the tuple $\bm{x}=(x_1,x_2,x_3,x_4)\in\mathbb R^4$ generate a non-periodic orbit under Equation~(\ref{Eq_main}). Let $\mathcal{U}=\mathcal{U}(\bm{x})$ be an arbitrary neighbourhood of $\bm{x}$. Then, there are tuples in $\mathcal{U}$ that generate periodic sequences of arbitrarily large period.
\end{proposition}
\begin{proof}
Without loss of generality, we can assume that the tuple $\bm{x}$ satisfies the conditions of Case $C_4$ [this is possible due to Claim~\ref{Claim0} and to the movement followed by the tuples, summarized in Figure~\ref{Diagrama}; also, notice that the infinite loops $C_1C_1C_1\ldots$ and $C_3C_3C_3\ldots$ are not admissible because, according to Table~\ref{Table_cases}, at some moment we will find a positive integer $n$ such that $x_2+n(x_3-x_4)$ is less than $x_4$ (so, we leave the Case $C_1$) or $x_2+n(x_3-x_4)$ is less than $x_3$ (and we leave the Case $C_3$)]. Therefore, any route in Figure~\ref{Diagrama} visits Case $C_4$ infinitely many times. 

By commodity in the notation, let us write the tuple $\bm{x}$ as $\bm{x}=(x,y,z,w)$. 
From Proposition~\ref{P:noper}, it must be fulfilled that $\frac{x}{w-z}\in\mathbb R\setminus\mathbb Q.$ 

Being $\bm{x}$ a tuple satisfying the inequalities of Case $C_4$, we have $x\geq w\geq y\geq z\geq 0$. Additionally, $w-z>0$, otherwise we would have the tuple $(x,w,w,w)$, which generates a $11$-periodic sequence as its initial conditions would be monotonic
(see~\cite{Gol}, where it is proved that, in fact, all the solutions of the general Equation~(\ref{e:intro}) are periodic of period $3k-1$ whenever the initial conditions are monotonic). 
 
In fact, we can even take another tuple $\widetilde{\bm{x}}$ in $\mathcal{U}$, arbitrarily close to $\bm{x}$, such that 
$\widetilde{\bm{x}}=(\widetilde{x}, \widetilde{y}, \widetilde{z}, \widetilde{w})$ verifies $\widetilde{w}-\widetilde{z}\neq 0,$ $\widetilde{x}>\widetilde{w}>\widetilde{w}-\widetilde{z}$, and $\frac{\widetilde{x}}{\widetilde{w}-\widetilde{z}}\in\mathbb R\setminus \mathbb Q$, so it is not restrictive to assume that the same tuple $\bm{x}$ also satisfies $x>w>w-z>0.$

By Dirichlet's Theorem relative to Diophantine approximation\footnote{The mentioned Dirichlet's result reads as follows.
 \begin{theorem}(Dirichlet)
Given $\alpha\in\mathbb R$ and $N>1$, there exist integers $x,y$ with $1\leq y\leq N$ and 
$\left|\alpha y - x\right|<\frac{1}{N}.$
When $\alpha$ is irrational, there are infinitely many reduced fractions $\frac{x}{y}$ with 
$\left|\alpha -\frac{x}{y}\right|<\frac{1}{y^2}.$
\end{theorem}} (see \cite{Liou} for a precise statement and proof), for the irrational number $\frac{x}{w-z}$ and $\varepsilon >0$ sufficiently small, we find that there exist infinitely many reduced fractions $\frac{m}{n}$ such that $\left|\frac{x}{w-z}- \frac{m}{n}\right|<\frac{1}{n^2}<\frac{\varepsilon}{w-z}.$ Moreover, being $\frac{x}{w-z}>1$ we can take $\frac{m}{n}>1.$

On the other hand, it was proved in \cite[Proposition 3.7]{Lin} that a tuple of the form $(y_1,y_2,y_3,y_4)$, with $y_1\geq y_4\geq y_2\geq y_3\geq 0$ and $y_1=\frac{q-p}{p}(y_4-y_3)$, where $p$ and $q$ are positive integers with $q\geq 2p+1$ and $\gcd(p,q)=1$, produces a periodic sequence of period $10p+11q.$ 
Setting $p:=n$ and $q:=m+n$, we get $\frac{m}{n}=\frac{q-p}{p}$ (notice that $\gcd(p,q)=\gcd(m,n)=1$, and $q>2p$ because $\frac{m}{n}>1$), hence we can apply this result to the tuple $\left(\frac{m}{n}(w-z),y,z,w\right)$ to obtain that it gives a periodic sequence of period $10p+11q=10n+11(m+n)$. Finally, observe that the tuples $\left(\frac{m}{n}(w-z),y,z,w\right)$ are close to $\bm{x}$, so they belong to $\mathcal{U}$, and that they present arbitrarily large periods.
\end{proof}

Finally, let us mention that, as main tool for the proof of Theorem \ref{t:teoprincipal}, we will use a consequence of Kronecker's Theorem, which we state for the sake of completeness (the reader can consult its proof in \cite{Nil}). Recall that the notation $\{\cdot\}$ is meant the fractional part of a number, that is, $\{x\}=x-\left\lfloor x\right\rfloor$, where $\left\lfloor x\right\rfloor$ denotes the largest integer less than or equal to $x$. Notice also that for any $\alpha\in\mathbb R,$ we have $\{t+\alpha\}=\{\alpha\}$ for all $t\in\mathbb Z$, since in this case $\left\lfloor t+\alpha\right\rfloor=\left\lfloor t\right\rfloor+\left\lfloor \alpha\right\rfloor$.

\begin{theorem} \label{Kro_Th}
\textbf{(Kronecker's Theorem)} Let $\delta$ be an irrational number. Then, for each non-empty open subinterval $U$ of $[0,1]$, there is $m\in \mathbb{N}$ such that  $\{m\cdot \delta\}\in U$.
\end{theorem}

From here, we derive the following result whose statement will be used in our discussion. 

\begin{corollary} \label{Cor_kro}
Let $\delta_1$ be an irrational number and let $\delta_2$ be an arbitrary real number. The set $S_{\delta_1} = \big\{ \{s\delta_1+\delta_2\big\}: s \in \mathbb{N}\}$ is dense in $[0,1]$.
\end{corollary}

\section{Evolution of non-negative tuples starting from Case $C4$}\label{Sec_condU}

In this part, we consider initial conditions $(x,y,z,w)$ with $x,y,z,w \in \mathbb{R}$ verifying the relation $x\geq w \geq y \geq z \geq 0$ (therefore, the initial tuple $(x,y,z,w)$ starts in the case $C_4$; this choice does not imply a loss of generality in our study, as it was already shown in \cite{Lin}, and as will become clear later, after giving the description of the routes).  We have changed the notation, being exonerated of the writing of sub-indexes,  for having a shorter writing and a more comfortable reading. Recall that the orbit of the general case is described by the diagram of Figure \ref{Diagrama} (see also Table~\ref{Table_cases}). 

As in \cite{Lin}, we will consider Figure \ref{Diagrama} as an oriented graph, $G=(V,U)$, where $V=\{C_1,C_2,C_3,C_4,C_5\}$ is a finite set and $U\subset V \times V$. The elements of $V$ are the vertices of the oriented graph $G$ and each element $(C_i, C_j)\in U$ will be called an arrow from $C_i$ to $C_j$. A path that always visits the same vertex is called a loop (our graph $G$ only admits two loops, one in $C_1$, and another in $C_3$). A route is a circuit that visits each vertex once, except the possibility of having loops. We denote them by $R_i$. Our graph only admits the following routes:

$R_1: C_4 \rightarrow C_5 \rightarrow C_2 \rightarrow C_1 \rightarrow \ldots \rightarrow C_1 \rightarrow C_4$.

$R_2: C_4 \rightarrow C_5 \rightarrow C_2 \rightarrow C_4$.

$R_3: C_4 \rightarrow C_5 \rightarrow C_3 \rightarrow \ldots \rightarrow C_3 \rightarrow C_2 \rightarrow C_1 \rightarrow \ldots \rightarrow C_1 \rightarrow C_4$.

$R_4: C_4 \rightarrow C_5 \rightarrow C_3 \rightarrow \ldots \rightarrow C_3 \rightarrow C_2 \rightarrow C_4$.

In order to clarify the study developed, we have structured this section as follows: In Subsection \ref{Sec_first}, for the sake of completeness, we analyze the evolution of the first terms of an orbit through the different routes. In Subsection \ref{Sec_general}, we repeat the process developed in the previous section, but with a general tuple of the form $\big(x,tx+y-s(w-z),z,w\big)$ since this kind of terms will play an essential role in the proof of the denseness.

\subsection{First terms of an orbit through the routes} \label{Sec_first}

Now, in order to understand how the initial conditions $(x,y,z,w)$ evolve under each route $R_i$, with $i=1,2,3,4$, we are going to write the terms that appear in the orbit in each case. To do so, we will apply the reasoning developed in the proof of Proposition 3.1 in \cite{Lin}, where all the computations where made (we write in bold the non-positive terms). Also, we will focus on the linear combinations of the form $tx+y-s(w-z)$ that appear, since they will play an essential role in the proof of the denseness. 

$\bullet$ Route $R_1: C_4 \rightarrow C_5 \rightarrow C_2 \rightarrow C_1 \rightarrow \ldots \rightarrow C_1 \rightarrow C_4$. We start with $(x, y, z, w)$ in $C_4$. Then, the orbit follows as: $$\bm{w-x}, w-y, w-z, \bm{-z}, x-z, x+y-w-z, x-w, (x,z, w+z-y, w) \ \text{in} \ C_5, $$ $$\bm{w-x}, w-z, y-z, \bm{-z}, x-z, x-w, x-y, (x,z,w,y) \ \text{in} \ C_2, $$
$$\bm{w-x}, w-z, \bm{-z}, w-z-y, x-z, x-w, \big(x, x+y-(w-z),z,w\big) \ \text{in} \ C_1,$$
$$\bm{y-(w-z)}, \bm{w-x-y+(w-z)}, w-z, \bm{-z}, -y + 2(w-z), x-z, x-w,$$ $$\big(x, x+y-2(w-z),z,w\big).$$ If this tuple verifies the conditions of $C_4$ we have finished the route. Otherwise, we will have a loop in $C_1$ and the orbit will follow as $$\bm{y-2(w-z)}, \bm{w-x-y+2(w-z)}, w-z, \bm{-z}, -y + 3(w-z), x-z, x-w,$$ $$\big(x,x+y-3(w-z),z,w\big).$$ Again, if the tuple is in $C_4$, we have finished the route, otherwise we will have another loop in $C_1$. Assume that we have $m_1\geq 0$ loops in $C_1$ (notice that $m_1<\infty$, because according to conditions of Figure~\ref{Diagrama}, at some point the condition $x_2+x_3\geq 2x_4$, or $x+y-n(w-z)+z\geq 2w$, will fail as $x+y-n(w-z)\stackrel{n}{\rightarrow}-\infty$). Then, the route will finish with the tuple $$\big(x, x+y-(m_1+2)(w-z),z,w\big) \ \text{in} \ C_4.$$

In the middle of the process we will have the tuples $$\big(x, x+y-j(w-z),z,w\big) \ \text{in} \ C_1 \ \text{with} \ j=1,\ldots, m_1+1.$$

Furthermore, every time that the orbit passes through $C_1$ the following non-positive terms will appear $$\bm{y-j(w-z)} \ \ \text{and} \ \ \bm{w - x - y + j(w-z)} \ \text{with} \ j=1,\ldots, m_1+1. $$

$\bullet$ Route $R_2: C_4 \rightarrow C_5 \rightarrow C_2 \rightarrow C_4$. We start with $(x, y, z, w)$ in $C_4$. Then, the orbit follows as: $$\bm{w-x}, w-y, w-z, \bm{-z}, x-z, x+y-w-z, x-w, (x,z, w+z-y, w) \ \text{in} \ C_5,$$ $$\bm{w-x}, w-z, y-z, \bm{-z}, x-z, x-w, x-y, (x,z,w,y) \ \text{in} \ C_2,$$
$$\bm{w-x}, w-z, \bm{-z}, w-z-y, x-z, x-w, \big(x, x+y-(w-z),z,w\big) \ \text{in} \ C_4.$$

Observe that the second term of the initial conditions, $y$, has evolved to $x + y -(w-z)$ under a route $R_2$. Moreover, the only non-positive terms that take place in $R_2$ are $w-x$ and $-z$.

$\bullet$ Route $R_4: C_4 \rightarrow C_5 \rightarrow C_3 \rightarrow \ldots \rightarrow C_3 \rightarrow C_2 \rightarrow C_4$. We start with $(x, y, z, w)$ in $C_4$. Then, the orbit follows as: $$\bm{w-x}, w-y, w-z, \bm{-z}, x-z, x+y-w-z, x-w, (x,z, w+z-y, w) \ \text{in} \ C_5,$$ $$\bm{w-x}, w-z, y-z, \bm{-z}, x-z, x-w, x-y, (x,z,w,y) \ \text{in} \ C_3,$$ $$\bm{w-x}, w-z, \bm{y-w}, \bm{-y+(w-z)}, x-z, x-w, x-y+(w-z), \big(x,z,w,y-(w-z)\big).$$
This tuple can verify either the case $C_2$ or $C_3$. If we have a loop in $C_3$, the orbit will continue as $$\bm{w-x}, w-z, \bm{y-(w-z)-w}, \bm{-y+2(w-z)}, x-z, x-w, x-y+2(w-z),$$ $$\big(x,z,w,y-2(w-z)\big).$$
Again, the new tuple satisfies either the conditions of case $C_2$ or those of case $C_3$. Assume that we have $m_3\geq 0$ loops in $C_3$ (as in route $R_1$, a similar reasoning gives $m_3<\infty$). Then, after that reiterative process, we will achieve the tuple $$\big(x,z,w, y -(m_3+1)(w-z)\big) \ \text{in} \ C_2.$$ Observe that in the process we have obtained the non-positive terms $$\bm{-y+j(w-z)} \ \ \text{and} \ \ \bm{y-(j-1)(w-z)-w} \ \text{with} \ j=1, \ldots, m_3+1.$$ Moreover, we have achieved the non-negative terms $y-j(w-z)$ with $j=1,\ldots, m_3+1$. 
Finally, if we continue computing the terms, we end going from $C_2$ to $C_4$ as follows:
$$\bm{w-x}, w-z, \bm{-z}, -y+(m_3+2)(w-z), x-z, x-w,$$ $$\big(x, x+y-(m_3+2)(w-z), z, w\big) \ \text{in} \ C_4.$$

$\bullet$ Route $R_3: C_4 \rightarrow C_5 \rightarrow C_3 \rightarrow \ldots \rightarrow C_3 \rightarrow C_2 \rightarrow C_1 \rightarrow \ldots \rightarrow C_1 \rightarrow C_4$. The terms appearing in the evolution of this route only contain a combination of elements of the routes $R_1$ and $R_4$, and the analysis will be omitted.

\begin{remark} \label{Remark1}
Notice that, independently of the route $R_i$, the initial conditions $(x,y,z,w)$ have evolved to $\big(x,x+y-n(w-z),z,w\big),$ where $n\in\mathbb{N}$ and $x+y-n(w-z)\geq 0$. Moreover, in the middle of the process we have obtained the non-negative terms $y-j(w-z)$ with $j=1,\ldots,n-1$. 
\end{remark}

\subsection{Evolution of a general tuple through the routes} \label{Sec_general}

From the results in Section \ref{Sec_first} we know that when an orbit of Equation \eqref{Eq_main} evolves through the routes $R_i$ there appear \emph{general tuples} of the form $\big(x,tx+y-s(w-z),z,w\big)$ where $t,s\in\mathbb{N}$ and $s\geq t$.  
In this subsection, with the intention of clarifying which terms appear in the orbit of general tuples, we are going to describe the routes $R_i$ again, but now when we start with such a tuple. The content of the section is instrumental and we will use it in the proof of Theorem \ref{t:teoprincipal}.

\medskip

$\bullet$ Route $R_1: C_4 \rightarrow C_5 \rightarrow C_2 \rightarrow C_1 \rightarrow \ldots \rightarrow C_1 \rightarrow C_4$. Let us consider the tuple $\big(x,tx+y-s(w-z),z,w\big)$ in $C_4$. Then, the orbit will evolve as follows: $$\bm{w-x}, w-tx-y+s(w-z), w-z, \bm{-z}, x-z, (t+1)x+y-s(w-z)-z-w, x-w $$ $$\big(x,z,w+z-tx-y+s(w-z), w\big) \ \text{in} \ C_5,$$ $$\bm{w-x}, w-z, -z+tx+y-s(w-z), \bm{-z}, x-z, x-w, (1-t)x-y+s(w-z), $$ $$\big(x,z,w,tx+y-s(w-z)\big) \ \text{in} \ C_2,$$ $$\bm{w-x}, w-z, \bm{-z}, -tx-y+(s+1)(w-z), x-z,x-w,$$ $$\big(x,(t+1)x+y-(s+1)(w-z),z,w\big) \ \text{in} \ C_1, $$ $$ \bm{tx+y-(s+1)(w-z)}, \bm{w-(t+1)x-y+(s+1)(w-z)}, w-z, \bm{-z},$$ $$-tx-y+(s+2)(w-z), x-z, x-w,$$ $$\big(x,(t+1)x+y-(s+2)(w-z), z,w\big).$$

This last tuple can either verify the conditions of $C_4$, and we would have ended the route, or verify again $C_1$. Let us assume that we have $m_1\geq 0$  loops in $C_1$  (an argument analogous to that of Section \ref{Sec_first} shows that $m_1<+\infty$), then we will have the tuples $$\big(x, (t+1)x+y-(s+j)(w-z), z,w\big) \ \text{with} \ j=1,\ldots, m_1+1,$$ verifying the case $C_1$ and we will end the route with $$\big(x,(t+1)x+y - (s+m_1+2)(w-z),z,w\big) \ \text{in} \ C_4.$$ Moreover, we emphasize that in that process it will appear the following non-positive terms
 $$\bm{tx+y-(s+j)(w-z)} \ \ \text{and} \ \ \bm{w-(t+1)x-y+(s+j)(w-z)}, \ j=1, \ldots, m_1+1.$$

$\bullet$ Route $R_2: C_4 \rightarrow C_5 \rightarrow C_2 \rightarrow C_4$. Let us consider the tuple $\big(x,tx+y-s(w-z),z,w\big)$ in $C_4$. Then, the orbit will evolve as follows: $$\bm{w-x}, w-tx-y+s(w-z), w-z, \bm{-z}, x-z, (t+1)x+y-s(w-z)-z-w, x-w $$ $$\big(x,z,w+z-tx-y+s(w-z), w\big) \ \text{in} \ C_5,$$ $$\bm{w-x}, w-z, -z+tx+y-s(w-z), \bm{-z}, x-z, x-w, (1-t)x-y+s(w-z), $$ $$\big(x,z,w,tx+y-s(w-z)\big) \ \text{in} \ C_2,$$ $$\bm{w-x}, w-z, \bm{-z}, -tx-y+(s+1)(w-z), x-z,x-w,$$ $$\big(x,(t+1)x+y-(s+1)(w-z),z,w\big) \ \text{in} \ C_4.$$

Now, the non-negative linear combination $tx+y-s(w-z)$ has evolved to $(t+1)x+y-(s+1)(w-x)$ under a route $R_2$. Furthermore, the only non-positive terms that take place in $R_2$ are $w-x$ and $-z$.

$\bullet$ Route $R_4: C_4 \rightarrow C_5 \rightarrow C_3 \rightarrow \ldots \rightarrow C_3 \rightarrow C_2 \rightarrow C_4$. Let us consider the tuple $(x,tx+y-s(w-z),z,w)$ in $C_4$. Then, the orbit will evolve as follows: $$\bm{w-x}, w-tx-y+s(w-z), w-z, \bm{-z}, x-z, (t+1)x+y-s(w-z)-z-w, x-w $$ $$\big(x,z,w+z-tx-y+s(w-z), w\big) \ \text{in} \ C_5,$$ $$\bm{w-x}, w-z, -z+tx+y-s(w-z), \bm{-z}, x-z, x-w, (1-t)x-y+s(w-z),$$ 
\begin{equation}\label{e:enC3}
\big(x,z,w,tx+y-s(w-z)\big) \ \text{in} \ C_3,
\end{equation}
\begin{equation}\label{e:enC32}
\bm{w-x}, w-z, \bm{tx+y-s(w-z)-w}, \bm{-tx-y+(s+1)(w-z)},
\end{equation}
$$x-z, x-w, (1-t)x-y+(s+1)(w-z), \big(x,z,w, tx+y-(s+1)(w-z)\big).$$

Now, we can be either in $C_2$ or in $C_3$. Assume that we have $m_3\geq 0$ loops in $C_3$ (as before, is easy to check that $m_3<+\infty$). Then, the reiterative process in $C_3$ will end with the tuple $$\big(x,z,w, tx+y-(s+m_3+1)(w-z)\big) \ \text{in} \ C_2.$$ Apart from this tuple, in the middle, after each loop we would have obtained the tuples $$\big(x,z,w, tx+y-(s+j)(w-z)\big) \ \text{in} \ C_3 \ \text{with} \ j=1,\ldots, m_3.$$ Moreover, it should be highlighted that we would have achieved the non-positive terms $$\bm{tx+y-(s+j)(w-z)-w} \ \text{and} \ \bm{-tx-y+(s+j+1)(w-z)} \ \text{with} \ j=0,\ldots,m_3.$$

Next, once we have $\big(x,z,w, tx+y-(s+m_3+1)(w-z)\big) \ \text{in} \ C_2$, the orbit continues as $$\bm{w-x}, w-z, \bm{-z}, -tx-y+(s+m_3+2)(w-z), x-z, x-w, $$ $$\big(x,(t+1)x+y-(s+m_3+2)(w-z), z, w\big) \ \text{in} \ C_4.$$

$\bullet$ Route $R_3: C_4 \rightarrow C_5 \rightarrow C_3 \rightarrow \ldots \rightarrow C_3 \rightarrow C_2 \rightarrow C_1 \rightarrow \ldots \rightarrow C_1 \rightarrow C_4$.  The terms appearing through the evolution of this route are a combination of the elements appearing in the routes $R_1$ and $R_4$, and we will omit the analysis.

\begin{remark} \label{Remark2}
Notice that, independently of the route $R_i$, the tuple $\big(x,tx+y-s(w-z),z,w\big)$ has evolved to $\big(x,(t+1)x+y-(s+n)(w-z),z,w\big)$ where $n\in\mathbb{N}$ and $(t+1)x+y-(s+n)(w-z)\geq 0$. Moreover, in the middle of the process we have obtained the non-negative terms $tx+y-(s+j)(w-z)$ with $j=1,\ldots,n-1$. 
\end{remark}

\section{Proof of Theorem~\ref{t:teoprincipal}}\label{S:mainproof}
We divide the proof of our main result into two parts: one devoted to analyze the accumulation points obtained by the non-negative terms of the solution $(x_n)$; and the second one concerned with the non-positive terms of the solution and their corresponding set of accumulation points. The union of the two cases will give the final proof of Theorem~\ref{t:teoprincipal}.
\subsection{Density of the non-negative terms} \label{Sec_positive}

After the detailed description of the routes that has been made in the previous subsections, and taking into account Remarks \ref{Remark1} and \ref{Remark2}, we obtain the following result:

\begin{lemma}\label{Claim}
Let $(x,y,z,w)$ be initial conditions with $x,y,z,w \in \mathbb{R}$ verifying the relation $x\geq w \geq y \geq z \geq 0$.  For every $s\in\mathbb{N}$, there exists at least a $t\in \mathbb{N}$ with $t\leq s$, such that the linear combination $tx+y-s(w-z)\geq 0$ belongs to the orbit generated by such initial conditions under Eq.~(\ref{Eq_main}).
\end{lemma}

Our goal is to prove that those non-negative terms $tx+y-s(w-z)$ that appear in the orbit are dense in the interval $[0,x]$. Notice that, since $x\geq tx+y-s(w-z) \geq 0$ for every linear combination that we are considering,  these terms belong to the interval $[0,x]$. 

Next,  we start dividing by $x>0$ the terms $tx + y - s(w-z)$ to simplify the analysis. So, we will study the accumulation of $t + \alpha - s\sigma$, where $\alpha = \frac{y}{x}$ and $\sigma = \frac{w-z}{x}.$ Notice that from Proposition~\ref{P:noper} it must hold 
\begin{equation} \label{Eq:sigma}
\sigma = \frac{w-z}{x}\in \mathbb{R} \setminus \mathbb{Q},  
\end{equation}
in order to not achieve periodicity, because otherwise at some moment we would find values of $s,s',t,t'$ for which $t - s\sigma=t' - s'\sigma$, and we will repeat the corresponding tuples, thus arriving to a periodic orbit.

Since $t + \alpha - s\sigma \in [0,1]$, it implies that  $t + \alpha - s\sigma = \{t + \alpha - s\sigma\}=$ $\{\alpha - s\sigma\}$.  By  Lemma~\ref{Claim}, $s$ goes through the natural numbers, hence by Corollary~\ref{Cor_kro} we get the density of the set of elements $t + \alpha - s\sigma$  in $[0,1]$ and, therefore, the density of the non-negative terms of an orbit in $[0,x]$.

\subsection{Density of the non-positive terms} \label{Sec_negative} 

In this section we prove the density of the non-positive terms in the interval $\big[\min\{-z,w-x\},0\big].$ Firstly, we gather in Table \ref{Table_negative} the non-positive terms that appear in each route $R_i$, $i=1,2,3,4$. This terms have been displayed in subsections \ref{Sec_first} and \ref{Sec_general}, and  obtained from the inspection of the proof of Proposition 3.1 in \cite{Lin}.

\begin{table}[ht] 
\centering
\begin{tabular}{@{}cccc@{}}
\toprule
\textbf{$R_1:$} & \begin{tabular}[c]{@{}c@{}} $w-x$ and $-z$\\ $tx + y - (s+j)(w-z)$ \\ $w - (t+1)x - y + (s+j)(w-z) $ \end{tabular} & $j\in \{1,\ldots, m_1+1\}$  \\ \midrule

\textbf{$R_2:$} & $w-x$ and $-z$  &  \\ \midrule

\textbf{$R_3:$} & The non-positive terms that appear in $R_1$ and $R_4$ \\ \midrule

\textbf{$R_4:$} & \begin{tabular}[c]{@{}c@{}} $w-x$ and $-z$ \\ $tx+y-(s+j-1)(w-z) - w$ \\ $-tx-y+(s+j)(w-z)$ \end{tabular} & $j\in\{1,\ldots, m_3+1\}$ \\  \bottomrule
\end{tabular}
\caption{Non-positive terms in the routes $R_i$.}
\label{Table_negative}
\end{table}

Next result establishes the bounds of the non-positive terms of a solution.
\begin{lemma} \label{L:claim2}
Given initial conditions $(x,y,z,w)$ with $x,y,z,w \in \mathbb{R}$ verifying the relation $x\geq w \geq y \geq z \geq 0$, then every non-positive term appearing in the corresponding orbit belongs to the interval $$\big[\min\{-z,w-x\},0\big].$$ 
\end{lemma}
\begin{proof}
On the one hand, the non-positive terms $tx+y-(s+j)(w-z),$ with $j=1,\ldots, m_1+1,$ appear while going from $$\big(x, (t+1)x+y-(s+j)(w-z),z,w\big) \ \text{in} \ C_1$$ to $$\big(x, (t+1)x+y-(s+j+1)(w-z),z,w\big).$$  Since, $\big(x, (t+1)x+y-(s+j)(w-z),z,w\big)$ verifies the case $C_1$, it yields to (see Table~\ref{Table_cases}) $$x \geq (t+1)x + y - (s+j)(w-z) \geq w,$$ or, equivalently, $$0\geq tx + y - (s+j)(w-z) \geq w-x.$$ So, the non-positive terms that appear when the orbit passes through $C_1$ belong to $[w-x,0]$. Moreover, it is easy to check that $tx+y-(s+j)(w-z)$ and $w-(t+1)x-y+(s+j)(w-z)$ are symmetric in the interval $[w-x,0]$ for every $j=1,\ldots, m_1+1$.

On the other hand, the non-positive terms $tx+y-(s+j-1)(w-z)-w$, with $j=1,\ldots,m_3+1$, appear while going from $$\big(x,z,w,tx+y-(s+j-1)(w-z)\big) \ \text{in} \ C_3$$ to $$\big(x,z,w,tx+y-(s+j)(w-z)\big).$$ Since, $\big(x,z,w,tx+y-(s+j-1)(w-z)\big)$ verifies the conditions from $C_3$, this means that (see again Table~\ref{Table_cases})  $$w \leq z + tx+y-(s+j-1)(w-z),$$ or $$-z \leq tx+y-(s+j-1)(w-z)-w \leq 0.$$ Hence, the non-positive terms that appear when the orbit passes through $C_3$ belong to the interval $[-z,0]$. Furthermore, it is easy to check that the linear combinations $tx+y-(s+j-1)(w-z) - w$ and $-tx-y+(s+j)(w-z)$ are symmetric in the interval $[-z,0]$. 
\end{proof}

Now, we will see that, in fact, the non-positive terms that appear in the orbit are dense in such interval $\big[\min\{w-x,-z\},0\big]$. To see that, we will proceed in three steps:
\begin{enumerate}
    \item We will see that the orbit of a non-periodic solution of Eq.~(\ref{Eq_main}) passes through the five cases $C_i$ an infinite number of times.

    \item We will study the non-positive terms that appear when the orbit passes through $C_1$ (routes $R_1$ and $R_3$) and we will prove that they are dense in $[w-x,0]$.
    \item We will focus on the non-positive terms that appear when the orbit passes through $C_3$ (routes $R_4$ and $R_3$) in order to see that they are dense in $[-z,0]$.
\end{enumerate}

Then, we will be able to gather those results to prove the density of the non-positive terms in the interval $\big[\min\{w-x,-z\},0\big]$.

\subsubsection{Evolution of a  non-periodic orbit through the cases $C_i$}

Now, we will see that the orbit of a non-periodic solution of Eq. (\ref{Eq_main}) must satisfy certain restrictions in the movement of routes.

Firstly, observe that the orbit cannot be an infinite concatenation of routes $R_2$. Indeed, if we start with initial conditions $(x,y,z,w)$ verifying the conditions of $C_4$, that is, $x\geq w \geq y \geq z \geq 0$, after a route $R_2$ we will have the tuple $\big(x, x + y - (w-z),z,w\big)$ satisfying $x \geq w \geq x + y - (w-z) \geq z \geq 0$. If we have another route $R_2$, we will obtain $\big(x, 2x + y - 2(w-z),z,w\big)$ with $x \geq w \geq 2x + y - 2(w-z) \geq z \geq 0$. Observe that the only change that takes place is in the second term, where we are adding the non-negative constant $x - (w-z)$. Therefore, there exists an $N\in \mathbb{N}$ such that $Nx + y -N(w-z) > w$, which contradicts the conditions of $C_4$.

This implies that, apart of the cases $C_2, C_4$ and $C_5$, the orbit must also pass through $C_1$ or $C_3$. We proceed to see that, in fact, the orbit travels through the Cases $C_1$ and $C_3$ infinitely many times for each one of them.

\begin{proposition} \label{P:claimC3C1}
Assume that the tuple of initial conditions $(x, y, z, w)$ verifies the conditions of Case $C_4$. Then, the orbit must verify the cases $C_1$ and $C_3$ infinitely many times.
\end{proposition}
\begin{proof}

To see it, we will argue by contradiction. Firstly, we will analyze the case of an orbit that, after certain iteration, does not pass through $C_1$ and, finally, we will study the case where the orbit, after a certain iteration, does not pass through $C_3$. In both cases, we will derive the corresponding contradiction.
 
Assume that, after a certain iteration, the orbit does not pass through $C_1$.  
Now, we will use that after each route ($R_2$ or $R_4$) the tuple verifying the case $C_4$, $\big(x, tx+y-s(w-z),z,w\big)$, satisfies (recall the conditions in Table~\ref{Table_cases}) $$w \geq tx+y-s(w-z),$$ \noindent or, equivalently, 
\begin{equation} \label{Ineq3}
    s \geq \frac{tx+y-w}{w-z}.
\end{equation}

On the other hand, if we go backwards in the orbit, taking into account that we are excluding the fact of passing through $C_1$, the tuple $\big(x, tx+y-s(w-z),z,w\big)$ in $C_4$ derives from the evolution of a tuple $\big(x,z,w,(t-1)x+y-(s-1)(w-z)\big)$, which belongs to the Case~$C_2$ (this can be easily seen from the inspection of Table \ref{Table_cases}). Thus, since the terms of the tuples verifying a certain case $C_i$ are non-negative, we have 
 $$(t-1)x+y-(s-1)(w-z) \geq  0,$$ \noindent so, 
\begin{equation} \label{Ineq4}
    s \leq \frac{(t-1)x+y}{w-z} +1.
\end{equation}

Hence, from (\ref{Ineq3}) and (\ref{Ineq4}), we achieve that
\begin{equation} \label{Ineq_joint2}
    \frac{tx+y-w}{w-z} \leq s \leq \frac{(t-1)x+y}{w-z} +1.
\end{equation}

From here, we will deduce that for every $t\in \mathbb{N}$, there exists an integer $s$ in the interval $$I_1 := \left[\frac{tx+y-w}{w-z}, \frac{(t-1)x+y}{w-z} +1 \right].$$

Observe that the length of such interval is $|I_1| = 1 - \frac{x-w}{w-z} \leq 1$. 
Observe also that at most  there exists a value $\tilde{t}$ for which 
$$\alpha=\frac{(\tilde{t}-1)x+y}{w-z}$$ is an integer number. Indeed, if $\alpha$ is an integer number, then $$ \frac{(\tilde{t}+p-1)x+y}{w-z}=\alpha+p\frac{x}{w-z}$$ must be irrational for any $p\in\mathbb{Z}\setminus\{0\}$ according to Proposition \ref{P:noper}, since $x/(w-z)\in\mathbb{R}\setminus\mathbb{Q}$ (recall Equation~\eqref{Eq:sigma}). So, for all $t=\tilde{t}+p$ with $p\in\mathbb{Z}\setminus\{0\}$, the term $\frac{(t-1)x+y}{w-z}$ is an irrational number.
	
To guarantee the existence of an integer $s \in I_1$, taking into account that $\frac{(t-1)x+y}{w-z}$ is not an integer number for all $t > \tilde{t}$, we need to ensure that $$\left\{\frac{(t-1)x+y}{w-z}\right\} \leq 1- \frac{x-w}{w-z},$$ where, as usual, $\{\cdot\}$ represents the fractional part of a number.

Indeed, if $A:= \frac{(t-1)x+y}{w-z}$ is not an integer number, then $\lfloor A+1\rfloor \in [A, A+1]$ where $\lfloor\cdot \rfloor$ represents the integer part of a number. So, $\lfloor A+1\rfloor \in I_1$ if and only if 
$\frac{tx+y-w}{w-z}\leq \lfloor A+1\rfloor=A+1-\{A\}$, hence
$$\{A\} \leq A- \frac{tx+y-w}{w-z}+1=1 - \frac{x-w}{w-z}.$$

Since $\frac{x}{w-z}$ is an irrational number, by Corollary \ref{Cor_kro}, the set $\left\{\left\{\frac{tx+y}{w-z} \right\}\right\}_{t\in \mathbb{N}, t\geq \tilde{t}}$ is dense in $[0,1]$. Therefore, it will exist  $t \in \mathbb{N}$ such that 
$\left\{\frac{(t-1)x+y}{w-z} \right\}>1-\frac{x-w}{w-z}$ and there will not exist the corresponding natural number~$s$.

To sum up, inequality (\ref{Ineq_joint2}) cannot hold for every $t\in \mathbb{N}$ and, consequently the route must pass through the case $C_1$. Furthermore, it must pass an infinite number of times, since otherwise, after the last time that it passes, we can apply the same argument to achieve a contradiction. 

\bigskip

Finally, assume that, after a certain iteration, the orbit does not pass through $C_3$, this means that it is composed eventually by the concatenation of routes $R_1$ and $R_2$. After each route, the tuple verifying the case $C_4$ will be of the form $\big(x, tx+y-s(w-z),z,w\big)$ where $t,s\in \mathbb{N}$ and $s\geq t$ (see Section~\ref{Sec_general}). Such tuple will evolve to $\big(x,z, z+w-tx-y+s(w-z), w \big)$ in $C_5$ (see Table~\ref{Table_cases}). Now, since it does not pass through $C_3$, according to the inequalities that determine Figure \ref{Diagrama}, for every $t\in \mathbb{N}$, the third term of that tuple must be greater than or equal to the double of the second one.  So the following inequality must hold: $$w-z \geq  tx + y - s(w-z),$$
\noindent or, equivalently, 
\begin{equation} \label{Ineq1}
    s \geq \frac{tx+y}{w-z} - 1.
\end{equation}

On the other hand, since the tuple $\big(x, tx+y-s(w-z),z,w\big)$ is in $C_4$, it holds that (see the conditions for cases $C_j$ in Table~\ref{Table_cases}) $z \leq tx+y-s(w-z)$. So, 
\begin{equation} \label{Ineq2}
    s \leq \frac{tx+y-z}{w-z}.
\end{equation}

Therefore, from (\ref{Ineq1}) and (\ref{Ineq2}), we have that 
\begin{equation} \label{Ineq_joint1}
    \frac{tx+y}{w-z} - 1 \leq s \leq \frac{tx+y-z}{w-z}.
\end{equation}

Thus, for every $t\in \mathbb{N}$, there exists a non-negative integer $s$ in the interval
$$I_2:= \left[\frac{tx+y}{w-z} - 1, \frac{tx+y-z}{w-z} \right].$$

Notice that the interval $I_2$ has length $|I_2| = 1 - \frac{z}{w-z} \leq 1$.  Also realize that, arguing as before, we can see that the number $\frac{tx+y}{w-z}$ can be an integer, or even a rational, at most for a single value of $\tilde{t}$. 

We will see that  to guarantee the existence of an integer $s\in I_2$, taking into account that  $\frac{tx+y}{w-z}$ is not an integer number for $t>\tilde{t}$, it is necessary that  $$\frac{z}{w-z} \leq \left\{ \frac{tx+y}{w-z} \right\}.$$ \noindent  Indeed,  if $A:= \frac{tx+y}{w-z}$ is not an integer number then, clearly, $\lfloor A \rfloor \in [A-1,A]$. By construction, $\lfloor A \rfloor \in \left[A-1,A - \frac{z}{w-z}\right]$ if and only if $\lfloor A \rfloor \leq A - \frac{z}{w-z}$ or, equivalently, $\frac{z}{w-z} \leq  A -\lfloor A \rfloor = \{A\}$.

Now, since $\frac{x}{w-z}$ is an irrational number, in order to not achieve periodicity, by Corollary \ref{Cor_kro}, the set $\left\{\left\{ \frac{tx+y}{w-z} \right\}\right\}_{t\in\mathbb{N},t>\tilde{t}}$ is dense in $[0,1]$. Thus, it will exist a number $t\in \mathbb{N}$ such that $\frac{z}{w-z} > \left\{\frac{tx+y}{w-z}\right\}$ and, therefore, it will not exist the corresponding natural number~$s$.

To sum up, inequality (\ref{Ineq_joint1}) cannot hold for every $t\in \mathbb{N}$ and, consequently, the route must pass through the case $C_3$. Moreover, notice that it has to pass an infinite number of times, since otherwise, after the last time that it passes through $C_3$ we can apply the same argument to achieve a contradiction.
This ends the proof of Proposition~\ref{P:claimC3C1}. 
\end{proof}

Definitely, we have seen that the orbit of a non-periodic solution of Eq. (\ref{Eq_main}) generated by the tuple $(x,y,z,w)$, and holding the inequalities of Case $C_4$, must  pass through every case $C_i$ infinitely many times. 

\subsubsection{Density of the non-positive terms in $[w-x,0]$}\label{ss:negativosA}

We start studying the non-positive terms that appear when the orbit passes through $C_1$.
According to Proposition~\ref{P:claimC3C1}, the orbit passes through $C_1$ infinitely many times. Moreover, the tuple verifies the conditions of Case $C_4$, $x\geq w\geq y\geq z$. We will only focus on those terms of the form $tx+y-s(w-z)\leq 0$, since the other non-positive terms that appear in such case, $w-(t+1)x-y+s(w-z)$, are symmetric in $[w-x,0]$ (see the proof of Lemma~\ref{L:claim2}), once we would have proved the density of the first ones it will be enough.

Take the initial conditions $(x,y,z,w)$. Every time that a route $R_1$ or $R_3$ takes place, the orbit will pass through $C_1$ and, then, there will appear non-positive terms of the form $\tilde{t}x+y-\tilde{s}(w-z)\leq 0$ with $\tilde{t},\tilde{s}\in\mathbb{N}$. 
Let us consider the sequence formed by those type of non-positive terms, that is, 
\begin{equation}\label{9a}
\big(t_nx+y-s_n(w-z)\big)_n,
\end{equation}
\noindent where $t_n,s_n\in \mathbb{N}$, $s_n \geq t_n$ and $s_{n+1}>s_n$ for every $n\geq 1$. Here, the sequence of natural numbers $(s_n)_n$ is increasing but does not necessarily increase one by one, so we cannot apply, directly, Corollary \ref{Cor_kro} to achieve the density. Hence, to prove the density of sequence \eqref{9a} in $[w-x,0]$, we will follow the next steps: (a)  We will embed the terms of  \eqref{9a} into the more general sequence
\begin{equation}\label{Eq:constC1}\big(\tilde{t}_nx + y - n(w-z)\big)_n,
\end{equation} \noindent where every term $\tilde{t}_nx + y - n(w-z)$ will belong to $[-x,0]$. Once this general sequence is constructed, we will prove that: 
(b) The sequence \eqref{Eq:constC1} is dense on $[-x,0]$; and (c) the terms of \eqref{Eq:constC1}  that appear in  \eqref{9a} belong to $[w-x,0]$ while the other terms belong to $[-x,w-x]$, which will complete the proof.

\medskip

\noindent (a) Now we construct the sequence \eqref{Eq:constC1}.  Given $n\in \mathbb N$, we have two possibilities:
\begin{itemize}
\item If there exists a term in \eqref{9a} of the form $t_nx+y-n(w-z)$ (i.e. such that $s_n=n$), we set $\tilde{t}_n=t_n$.

\item Otherwise, if there is not such a term  $t_nx+y-n(w-z)$ in \eqref{9a}, by Lemma~\ref{Claim} we know that for every natural number $n$ there exists a number $t \in \mathbb{N}$ such that $x \geq tx + y - n(w-z) \geq 0$, so $(t-1)x +y - n(w-z)$ will belong to $[-x,0]$. In this situation, we take $\tilde{t}_n=t-1$.
\end{itemize}

\noindent (b) Once we have constructed the sequence~(\ref{Eq:constC1}), in order to apply Corollary \ref{Cor_kro}, we divide its terms by $x$, obtaining the associated $\tilde{t}_n+\alpha - n\sigma \in [-1,0]$, where $\alpha = \frac{y}{x}$ and $\sigma = \frac{w-z}{x}\in \mathbb{R}\setminus \mathbb{Q}$ (recall Equation~\eqref{Eq:sigma}). 

Since $\tilde{t}_n+\alpha - n\sigma \in [-1,0]$ 
and $\{\tilde{t}_n+\alpha -n\sigma\}= \left\{\alpha -n\sigma\right\}$, by Corollary \ref{Cor_kro} we get the density of 
$\left(\tilde{t}_n+\alpha - n\sigma\right)_n$ in $[-1,0]$ and, therefore, the density of the sequence $\big(\tilde{t}_nx + y - n(w-z)\big)_n$ in $[-x,0]$.

\noindent (c) We claim that every term of the subsequence \eqref{9a} belongs to $[w-x,0]$, while the rest of the terms of \eqref{Eq:constC1} belong to $[-x,w-x]$.
We have already seen the first claim at the beginning of the Section~\ref{Sec_positive}, in Lemma~\ref{L:claim2}, so we will only study those terms that do not appear in the orbit.

Let us consider a non-positive term $\tilde{t}_nx+y-n(w-z)$ which does not appear in the orbit. From our study in Section~\ref{Sec_general}, it is known that the positive term $(\tilde{t}_n+1) x+ y -n(w-z)=tx+ y -n(w-z)$ should be part of some of the positive tuples appearing in the evolution of the initial tuple $(x,y,z,w)$, namely: $(x,tx+y-n(w-z),z,w)$ in the Case $C_1$ or  $C_4$;  or $(x,z,w,tx+y-n(w-z)$ in the Cases $C_2$ or $C_3$ (Case $C_5$ is excluded). 

\begin{itemize}
\item If $\big(x, tx+y-n(w-z),w,z\big)$ is a tuple in Case $C_1$, then it is easily seen that $(t-1)x+y-n(w-z)=\tilde{t}_n+y-n(w-z)$ belongs to the orbit, contrarily to our hypothesis on the value $\tilde{t}_n$.
\item When $\big(x, z,w, tx+y-n(w-z)\big)$ is a tuple in Case $C_2$ or Case $C_3$, the conditions appearing in Table~\ref{Table_cases} impose, in particular, that $x_3\geq x_4,$ that is,  $w\geq tx+y-n(w-z)$, which is equivalent to $-x+w\geq (t-1)x + y -n(w-z)$, as desired.
\item For the tuple  $\big(x, tx+y-n(w-z),w,z\big)$ being in Case $C_4$, we take into account that now the conditions imposed, among others, that $x_4\geq x_2$, so $w\geq tx+y-n(w-z)$, which yields the desired inequality $-x+w\geq (t-1)x + y -n(w-z)$.
\end{itemize}

In conclusion, we can split the sequence $\big(\tilde{t}_nx + y - n(w-z)\big)_n$, which is dense in $[-x,0]$, in two subsequences: the first one, formed by the non-positive terms that appear in the orbit, which are in $[w-x,0]$; and the second one, formed by those that do not appear, which are in $[-x,w-x]$. Therefore, we can assure that the sequence of the non-positive terms that belong to the orbit, $\big(t_nx + y - s_n(w-z)\big)_n$, is dense in $[w-x,0]$.

\subsubsection{Density of the non-positive terms in $[-z,0]$}\label{ss:negativosB}

Finally, we will study the non-positive terms that appear when the orbit passes through the Case $C_3$ (we will use implicitly the fact that, from  Proposition~\ref{P:claimC3C1}, this happens an infinite number of times). In concrete, we will only focus on those of the form $tx+y-(s-1)(w-z)-w\leq 0$, since they are symmetric $[-z,0]$ with the other non-positive terms, $-tx-y+s(w-z)$, that appear in such case. So, once we would have proved the density of the first ones it will be enough.

We proceed in an analogous way as in Section \ref{ss:negativosA}. Take the initial conditions $(x,y,z,w)$. Every time that a route $R_4$ or $R_3$ take place, the orbit will pass through $C_3$ and then we will have non-positive terms of the form $\tilde{t}x+y-\tilde{s}(w-z)-w\leq 0$. Let us consider the sequence formed by those non-positive terms, that is, 
\begin{equation}\label{9b}
\big(t_nx+y-s_n(w-z)-w\big)_n,
\end{equation}
\noindent where $t_n,s_n\in \mathbb{N}$, $s_n \geq t_n$ and $s_{n+1}>s_n$ for every $n\geq 1$. Here, the sequence of natural numbers $(s_n)_n$ is increasing but does not increase one by one so, as before, we will consider that \eqref{9b} is a subsequence of a more general sequence of the form
\begin{equation}\label{9c}
\big(\hat{t}_nx + y - n(w-z)-w\big)_n,
\end{equation}
\noindent where every term $\hat{t}_nx + y - n(w-z)-w$ belongs to $[-x,0]$.  
We proceed it in a similar way than in the previous subsection. (a) Given $n\in \mathbb N$:
 \begin{itemize}
\item If there exists a term $t_nx + y - n(w-z)-w$ in \eqref{9b}
(i.e. $s_n=n$), then we set $\hat{t}_n=t_n$. Observe that from the computations developed in Section~\ref{Sec_general}, and stated in the expressions \eqref{e:enC3} and \eqref{e:enC32}, we can deduce that the term $t_nx + y - n(w-z)-w$ appears by iterating a 
 tuple of the orbit of the form $(x, z, w, tx + y - n(w - z))$ that belongs the case Case $C_3$.

\item Otherwise, if there is not such a term  $t_nx + y - n(w-z)-w$ in \eqref{9b}, then by Lemma~\ref{Claim}, we know that, for every natural $n$, there exists a $t\in\mathbb{N}$, such that $x\geq tx+y-n(w-z) \geq 0$. If $tx+y-n(w-z)-w < 0$, we set $\hat{t}_n=t$; else if $tx+y-n(w-z)-w \geq  0$, then $(t-1)x+y-n(w-z)-w\leq 0$ and we set $\hat{t}_n=t-1$; in both cases, clearly $\hat{t}_nx + y - n(w-z)-w\in[-x,0]$.
\end{itemize}

\noindent (b) We divide $\hat{t}_nx + y - n(w-z)-w$ by $x$ so each term reduces to $\hat{t}_n + \alpha - n\sigma - \delta$, where $\alpha =\frac{y}{x}$, $\sigma = \frac{w-z}{x} \in \mathbb{R}\setminus \mathbb{Q}$ and $\delta = \frac{w}{x}$. Thus, $-1 \leq \hat{t}_n + \alpha - n\sigma - \delta\leq 0,$ and we have that $\hat{t}_n + \alpha - n\sigma - \delta = \{\hat{t}_n + \alpha - n\sigma - \delta\}-1$. Then, by Corollary \ref{Cor_kro}, we obtain the density of  the sequence $\left(\hat{t}_n + \alpha - n\sigma - \delta\right)_n$ in $[-1,0]$ and therefore, the density of the sequence \eqref{9c} in $[-x,0]$.

\noindent (c) Next, we claim that every term of the subsequence \eqref{9b} belongs to $[-z,0]$, while the rest of the terms of \eqref{9b} belongs to $[-x,-z]$. The first part of the claim is consequence of our study in Lemma~\ref{L:claim2} in Section~\ref{Sec_general}, 
so we will only study those terms that do not appear in the orbit.
 
Let us consider a non-positive term $\hat{t}_nx + y - n(w-z)-w$ in \eqref{9c} which does not appear in the orbit. We will keep track the evolution of this term. First, remember that, according to the development in Section~\ref{Sec_general}, we have that a non-negative term of the form $tx+y-n(w-z)$ exists (by Lemma \ref{Claim}) and it appears in some of the following type of tuples: $(x,tx+y-n(w-z),z,w)$ in the Case $C_1$ or the Case $C_4$; or $(x,z,w,tx+y-n(w-z)$ in the Cases $C_2$ or $C_3$ (the Case $C_5$ is excluded). 
\begin{itemize}
\item  If $(x_1,x_2,x_3,x_4)=\big(x, tx+y-n(w-z),z,w\big)$ is a tuple in the Case $C_1$, then it holds the conditions $x_1\geq x_2 \geq x_4 \geq x_3\geq 0.$ These conditions imply that
$tx+y-n(w-z)-w\geq 0$ therefore, by the definition of \eqref{9c}, $\hat{t}_n=t-1$. Furthermore, the conditions in $C_1$ imply that
$(t-1)x+y-n(w-z)\leq 0$ and $w-z\geq 0$.  Hence,   
$(t-1)x+y-(n+1)(w-z)\leq 0$, but this inequality is equivalent to $(t-1)x+y-n(w-z)-w \leq -z$, which gives $\hat{t}_nx + y - n(w-z)-w\leq -z$, as we wanted to prove.
\item For a tuple  $(x_1,x_2,x_3,x_4)=\big(x, z,w, tx+y-n(w-z)\big)$ in the Case $C_2$ or the Case $C_3$, the conditions appearing in Table~\ref{Table_cases} impose, in particular, that $x_3\geq x_4,$ that is,  $w\geq tx+y-n(w-z)$, so $tx+y-n(w-z)-w\leq 0$, and therefore (again, by the definition of  \eqref{9c}) $\hat{t}_n=t$.
We need to show, then, that $\hat{t}_nx +  y - n(w-z)-w=tx+y-n(w-z) -w \leq -z$. In Case $C_2$, from Table~\ref{Table_cases} we know the following relation between the coordinates of the tuple, $x_2+x_4\leq x_3$, that is, $z+tx+y-n(w-z)\leq w$, which is equivalent to $tx+y-n(w-z)-w\leq -z,$ so we finish the case. For a tuple in the Case $C_3$, by the study of Section~\ref{Sec_general}, we would find by iteration the non-positive element $tx+y-n(w-z)-w$, contrary to the definition of $\hat{t}_n$, which requires the non-existence of such non-positive elements.
\item For a tuple  $(x_1,x_2,x_3,x_4)=\big(x, tx+y-n(w-z),z,w\big)$  in the Case $C_4$, we know that $x_2\leq x_4$, so $tx+y-n(w-z) - w\leq 0$, therefore $\hat{t}_n=t.$ 
 Moreover, the tuple evolves to $(\tilde{x}_1,\tilde{x}_2,\tilde{x}_3,\tilde{x}_4)=\big(x, z, w+z -tx - y+n(w-z), w\big)$ in $C_5$, where the conditions $\tilde{x}_1\geq \tilde{x}_4 \geq \tilde{x}_3 \geq \tilde{x}_2$ must hold. We need to prove 
$\hat{t}_nx + y - n(w-z)-w=tx+y-n(w-z) -w \leq -z$. Once we have arrived to the tuple in $C_5$, we have two options: (i) If we pass from the Case $C_5$ to the Case  $C_3$, we obtain the tuple $\big(x, z, w, tx+y-n(w-z)\big)$ in $C_3$ but, by the developments in Section~\ref{Sec_general},  we know that by the iteration of this tuple  we obtain the element 
$tx+y-n(w-z)-w$ which, as in the preceding case, is contrary to the definition of $\hat{t}_n$. (ii) If we pass from the Case $C_5$ to the Case  $C_2$,
we find the new tuple $\big(x,z,w,tx+y-n(w-z)\big)$ in $C_2.$  
Using the information collected in Figure \ref{Diagrama} and Table~\ref{Table_cases}, we deduce that this tuple in $C_2$ comes from a tuple
$$
(\bar{x}_1,\bar{x}_2,\bar{x}_3,\bar{x}_4)=\big(x,z,z+w-tx-y+n(w-z),w\big)\in C_5
$$ satisfying  $\bar{x}_3\geq 2\bar{x}_2$, which means
$$w+z-tx-y+n(w-z) \geq 2z,$$ or, equivalently, $-z \geq tx+y-n(w-z)-w.$  
\end{itemize}

In conclusion, we can split the sequence $\big(\hat{t}_nx+y-n(w-z)-w\big)_n$, which is dense in $[-x,0]$, in two subsequences: the first one, formed by the non-positive terms that appear in the orbit, that are in $[-z,0]$; and the second one, formed by those that do not appear, which are in $[-x,-z]$. Therefore, we can assure that the sequence of the non-positive terms that appear in the orbit, $\big(t_nx+y-s_n(w-z)\big)_n$ is dense in $[-z,0]$.

\subsubsection{Proof of Theorem~\ref{t:teoprincipal}} \label{Sec_main}

Given real initial conditions $(x_1,x_2,x_3,x_4)$, by Section~\ref{Sec_prelim} we know that the initial tuple is equivalent, in the sense of Definition~\ref{Def_equi}, to a tuple of non-negative terms $(x,y,z,w)$, with $x=\max\{x_n:n\geq 1\}$, and holding the inequalities characterizing Case $C_4$, namely, $x\geq w \geq y\geq z$. Assume that they generate a non-periodic orbit. Then, $\frac{w-z}{x}\in\mathbb R\setminus\mathbb Q$ from Proposition~\ref{P:noper} and $x>z$, otherwise the initial conditions are monotone, in fact, constant, and the sequence would be periodic, a contradiction.

Finally, from the study carried out in Subsection~\ref{Sec_positive},  and in the Subsections~\ref{ss:negativosA} and \ref{ss:negativosB},  we conclude that the solution $(x_n)$ is dense in a compact interval, to be more precise, it accumulates in $\big[\min\{w-x,-z\},x\big]$.

\section{Conclusions and other future directions} \label{Sec_conclusions}

With the main result of this paper, we have deciphered the dynamics of non-periodic orbits of Equation~(\ref{Eq_main}). This contribution has allowed us to complement the results appearing in \cite{Cso} and \cite{Lin}. The set of results of these publications along with the ones of the present work fully characterize the dynamics of Equation~(\ref{Eq_main}). Below, we describe some open problems and possible future lines of research. We also present a final result in which we give a new invariant for the equation.

It could be of some interest to extend this result to higher order difference equations of max-type, in particular, for the fifth order equation $$x_{n+5}=\max\{x_{n+4},x_{n+3},x_{n+2}, x_{n+1},0 \}-x_n.$$ Surely, the strategy of routes $R_j$ and Cases $C_j$ can be translated to this context, with suitable modifications, with the aim of knowing both the set of periods of the equation as well as the behaviour of non-periodic orbits. Another possible extension is to pay attention to the study of max-equations of order $k$ and type $x_{n+k}=\max\{x_{n+k-1},\ldots,x_{n+1}, A \}-x_n$, being $A$ an arbitrarily fixed real number. To this respect, for orders $k=2$ and $k=3$ partial results are found in \cite{BT} whenever $A=1$ or $A=-1$.

Another further line of investigation is related with the connection between difference equations and their associate discrete dynamical systems. Notice that Equation~(\ref{Eq_main}) can be viewed as a discrete dynamical system $X_{n+1}=F(X_n)$, where $F:\mathbb R^4\rightarrow \mathbb R^4$ is given by $$ F(x,y,z,w)=(y,z,w,\max(0,y,z,w)-x).$$
We have seen that the limit sets of Equation \eqref{Eq_main}, in non-periodic cases, are closed intervals. These intervals can be seen as projections of the limit sets of the corresponding trajectories, say $\gamma$, of the discrete dynamical system associated to $F$, starting from arbitrary initial condition $X\in\mathbb{R}^4$. 
 Then, a natural question arises: to determine the topological characteristics of the limit sets of its orbits $$\omega_{F}(\gamma)=\{Y\in\mathbb R^4: \exists\, (n_j)_j\subset \mathbb N, n_1<n_2<\ldots<n_j<\ldots, \lim_{j\rightarrow \infty}F^{n_j}(X)=Y\}.$$  
Our numerical simulations suggest that these trajectories could densely fill closed curves of $\mathbb{R}^4$. Furthermore, these curves could be simple (without self-intersections). It would be of some interest to study the topology of these limit sets or prove, at least, that $\omega_{F}(\gamma)$ is a connected set.   

We present a couple of examples. In Figure \ref{f:1} we show a pair of views of a projection into $\mathbb{R}^3$ of $10^4$ iterates of map $F$ for the initial conditions $(2\sqrt{2},2,0,1)$. Since $(2\sqrt{2},2,0,1)$ is equivalent to the new tuple $(2\sqrt{2},1,0,1)$ in Case $C_4$, as a direct computation shows, Theorem~\ref{t:teoprincipal} establishes that the solution $(x_n)$ of Eq.~(\ref{Eq_main}) accumulates in the interval
$[1-2\sqrt{2},2\sqrt{2}]$. The  inspection of the iterates of $F$ with this initial condition indicates that the iterates fill a closed curve. The self-intersections that appear could be due to an effect of the projection of the curve into $\mathbb{R}^3$.
\begin{figure}[ht] 
\centering
\includegraphics[scale=0.34]{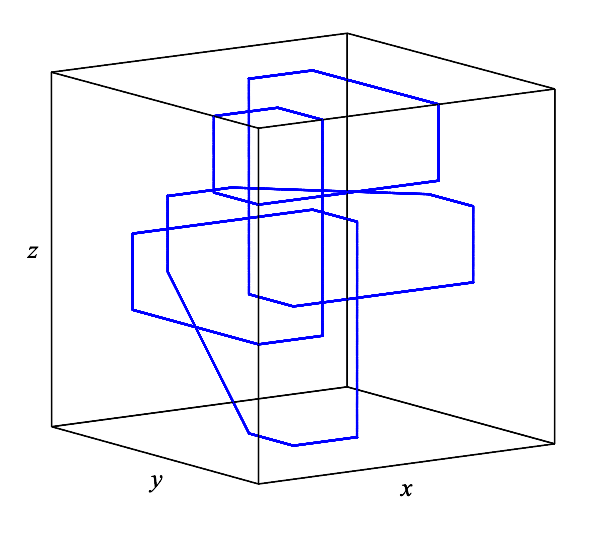}\includegraphics[scale=0.34]{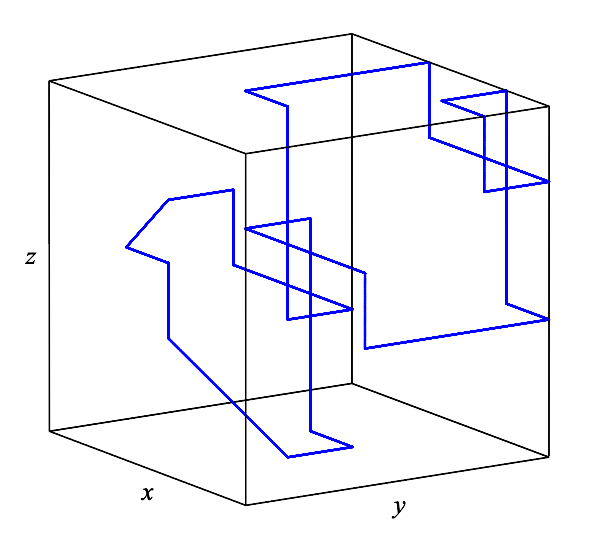}
\caption{Two views of the orbit of the map $F$ with initial conditions $(2\sqrt{2},2,0,1)$}\label{f:1}
\end{figure}

In Figure~\ref{f:2} we show some projections in $\mathbb{R}^3$ of $10^4$ iterates of map $F$ for the initial conditions $(\sqrt{2}+10\sqrt{3},1,2,0)$. Here, by a straightforward calculation it is easily seen that the initial tuple is equivalent to $(\sqrt{2}+10\sqrt{3},\sqrt{2}+10\sqrt{3}-17,1,2)$, a tuple in Case $C_4$, therefore from Theorem~\ref{t:teoprincipal}, the sequence $(x_n)$ generated by Eq.~(\ref{Eq_main}) accumulates in the interval $\left[2-\sqrt{2}-10\sqrt{3},\sqrt{2}+10\sqrt{3}\right]$.
The inspection of this projection into $\mathbb{R}^3$ indicates that it is a simple closed curve.
\begin{figure}[ht] 
\centering
\includegraphics[scale=0.34]{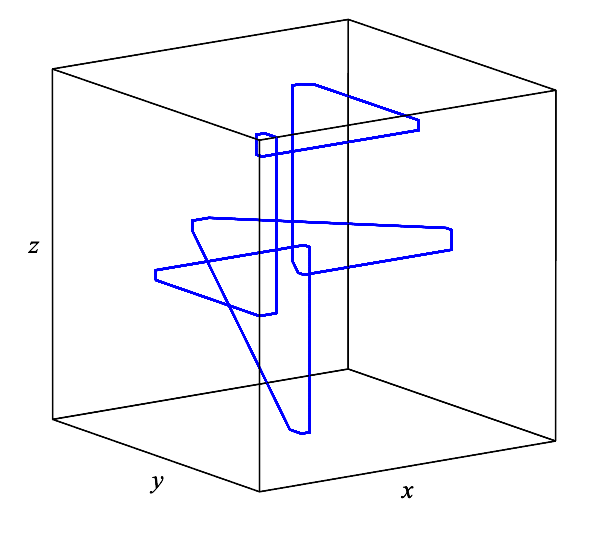}\includegraphics[scale=0.34]{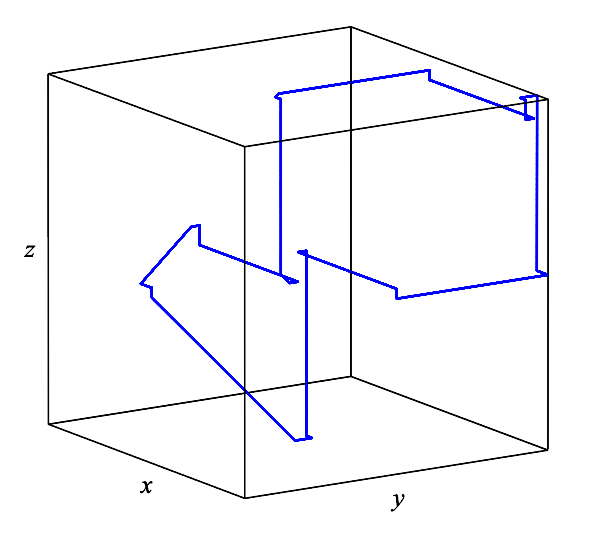}
\caption{Two views of the orbit of the map $F$ with initial conditions $(10\sqrt{3}+\sqrt{2},1,2,0)$}\label{f:2}
\end{figure}

Recall that a \emph{first integral} of the discrete dynamical system in $\mathbb{R}^n$ generated by a map $G$
is a non constant function in a nonempty open set $\mathcal{U}\subseteq \mathbb{R}^n$, $V:\mathcal{U}\rightarrow \mathbb{R}$, which is constant on
the orbits, i.e.
$$ V(G(\mathbf{x}))=V(\mathbf{x}) \mbox{ for all } \mathbf{x}\in{\cal{U}}.$$
A set $V_1,\ldots,V_k$ of first integrals of $G$ defined in an open set $\mathcal{U}$ are  \emph{functionally dependent} if there exists a real-valued function $R:\mathcal{U}\rightarrow \mathbb{R}$ not identically zero such that $R(V_1(\mathbf{x}),\ldots,V_k(\mathbf{x}))=0$ for all $\mathbf{x}\in \mathcal{U}$. Otherwise, we say that they are \emph{functionally independent}, \cite[pp. 84--85]{Ol}.
Also,  we will say that $G$ is
\emph{completely integrable} if it has $n$ functionally
independent first integrals.

In \cite{BT} it is proved that Equation \eqref{Eq_main} has an invariant that gives rise to the first integral of $F$, 
\begin{align*}
V_1(x,y,z,w)=&\max  \left(0, x , y , z , w\right)+
\max \left(0, -x \right)+\max \left(0, -y \right)+\max  \left(0, -z \right)\\ 
&+\max \left(0, -w \right).
\end{align*}
Our simulations are compatible with the fact that the map $F$ could have exactly three functionally independent first integrals. 
Remember that the map $F$ cannot be completely integrable because it is not globally periodic, see \cite[Theorem 1(b)]{CGM}.
Based on the result of our simulations, we thought that was interesting to find new first integrals of $F$ (or invariants for \eqref{Eq_main}) functionally independent with $V_1$.
In this sense we prove
\begin{proposition}\label{p:V2}
The function 
\begin{align*}
V_2(x, y, z, w)=&\max (0, x, y, z, w, x + w) + \max(0, x, y) + \max(0, y, z) + \max(0, z, w)\\
&- x - y - z - w
\end{align*}
is first integral of the map $F$. In other words, it is an invariant function for the recurrence \eqref{Eq_main}. Furthermore, there exist nonempty open sets in $\mathbb{R}^4$ where $V_2$ is functionally independent with $V_1$.
\end{proposition} 

\begin{proof} To prove the result, we have to show that  $\Delta V_2=V_2(y,z,w,F(x,y,z,w))-V_2(x,y,z,w)\equiv 0$. A straightforward computation yields to
\begin{align*}
\Delta V_2=&\max\{0, y, z, w, \max\{y,z,w,0\}-x, \max\{y,z,w,0\}-x+y \}\\
&+\max\{0,w,\max\{y,z,w,0\}-x\}
-\max\{0,x,y,z,w,x+w\}\\
& -\max\{0,y,z,w\}- \max\{0,x,y\} +2x.
\end{align*}
Now, we prove that $\Delta V_2\equiv 0$ when  $x\geq 0$ and $x \geq \max\{0,y,z,w\}$. The rest of the cases (namely, $0\leq x\leq \max\{0,y,z,w\}$; $x<0$ and $x\geq \max\{0,y,z,w\}$; and $x<0$ and $x\leq \max\{0,y,z,w\}$) can be done analogously. 
Indeed, suppose that $x\geq 0$ and $x \geq \max\{0,y,z,w\}$ then,
\begin{itemize}
\item $\max\{0, y, z, w, \max\{y,z,w,0\}-x, \max\{y,z,w,0\}-x+y \}=\max\{0, y, z, w\},$
\item $\max\{0,w,\max\{y,z,w,0\}-x\}=\max\{0,w\},$
\item $\max\{0,x,y,z,w,x+w\}=\max\{x,x+w\},$
\item $\max\{0,x,y\}=x.$
\end{itemize}
Hence
$$
\Delta V_2=\max\{0,w\}-\max\{x,x+w\}+x.
$$
If $w\geq 0$, then $\Delta V_2=w-x-w+x=0$; if $w< 0$, then $\Delta V_2=x-x=0$. Therefore, $\Delta V_2\equiv 0$.

\medskip

To prove that there exists open sets in which $V_1$ and $V_2$ are functionally independent, consider (for instance) an initial condition in the open set
$\mathcal{U}=\{(x,y,z,w):\, x>w>y>z>0\}$, that is, 
satisfying conditions of case $C_4$ with strict inequalities. A computation shows that, in this case,
$$V_1(x,y,z,w)=x \mbox{ and } V_2(x,y,z,w)=x+w-z,$$
which are obviously functionally independent.
\end{proof}

We find it interesting to comment on how we have found the second invariant $V_2$. Note that Equation \eqref{Eq_main} is the ultradiscretization, in the sense of \cite{ON}, of the $4$th--order Lyness' Equation or, equivalently, that $F$ is the ultradiscretization of the $4$-dimensional Lyness map
$$
L_4(x,y,z,w)=\left(y,z,w,\frac{a+y+z+w}{x}\right).
$$
It is known that $L_4$ has two functionally independent first integrals \cite{GKI}, see also \cite{CGM08,TKQ}:
\begin{align*}
H_1 (x, y, z, w) &=\frac{(a + x + y + z + w)(x + 1)(y + 1)(z + 1)(w + 1)}{x y z w },\\
H_2 (x, y, z, w) &=\frac{ (a + x + y + z + w + xw)(1 + x + y)(1 + y + z)(1 + z + w)}{xyzw}.
\end{align*}
It can be seen that $V_1$ is the ultradiscrete version of $H_1$. We have obtained $V_2$ as the ultradiscretization of $H_2$.

Although, as far as we know, it is not proven, the known results are in good agreement with the fact that the conjecture expressed in \cite{GKI}, which states that the maximum number of functionally independent first integrals of the $k$-dimensional Lyness map $L_k$ is~$\lfloor(k+1)/2\rfloor$, is true (see \cite{BR,CGM08,TKQ}). 
If this was correct,  $L_4$ should not admit  more than $2$ functionally independent first integrals. As mentioned above, our numerical simulations are compatible with the fact that the map $F$ could have three functionally independent first integrals (compare the graphs of Figures \ref{f:1} and \ref{f:2} and the graph of the iterations of a map $L_4$ that is presented in Figure 1 of \cite{CGM08}, where it is clearly intuited that the iterates of $L_4$ evolve over a $2$-dimensional surface). The existence of a third first integral for the map $F$ would show that the maps that come from the ultradiscretization of non--globally periodic rational maps can have more first integrals than the original rational maps (remember that in \cite[Theorem 3.5]{ON} it is shown that globally periodic rational maps give rise to globally periodic ultradiscrete maps; in such a  case, both maps will have as many first integrals as the phase space, according to the results in \cite{CGM}). 

We leave open the possibility of finding a new functionally independent first integral for the map $F$. 

\section*{Acknowledgements}
This work has been supported by the grant MTM2017-84079-P funded by \\ MCIN/AEI/10.13039/501100011033 and by ERDF ``A way of making Europe'', by the European Union. The second author acknowledges the group research recognition 2021 SGR 01039 from Ag\`{e}ncia de Gesti\'{o} d'Ajuts Universitaris i de Recerca.

\end{document}